\documentclass{article}
\usepackage{CJK,CJKnumb,CJKulem,times,dsfont,ifthen,mathrsfs,latexsym,amsfonts,color}
\usepackage{amsmath,amsthm,makeidx,fontenc,amssymb,bm,graphicx,psfrag,listings,curves,extarrows,enumitem}
\usepackage{hyperref}

\usepackage{geometry}
\usepackage{indentfirst}
\usepackage{amssymb}
\makeatletter

\newcommand{\Rmnum}[1]{\expandafter\@slowromancap\romannumeral #1@}
\makeatother

\usepackage[showonlyrefs]{mathtools}  
\mathtoolsset{showonlyrefs=true}

\newtheorem{theorem}{Theorem}[section]
\newtheorem{lemma}{Lemma}[section]

\newtheorem{proposition}{Proposition}[section]
\newtheorem{remark}{Remark}[section]

\newcommand{\N}{{\mathbb N}}

\newcommand{\C}{{\mathbb C}}

\newcommand{\R}{{\mathbb R}}
\newcommand{\bean}{\begin{eqnarray*}}
	\newcommand{\eean}{\end{eqnarray*}}

\newcommand{\sbr}[1]{\left(#1\right)}
\newcommand{\mbr}[1]{\left[#1\right]}
\newcommand{\lbr}[1]{\left\{#1\right\}}

\setitemize{itemindent=38pt,leftmargin=0pt,itemsep=-0.4ex,listparindent=26pt,partopsep=0pt,parsep=0.5ex,topsep=-0.25ex}

\numberwithin{equation}{section}

\begin{document}
	\theoremstyle{plain}

	\title{\bf  Sign-changing  solution  for logarithmic elliptic  equations with critical exponent\thanks{Supported NSFC(No.12171265,11025106). E-mail addresses: liuth19@mails.tsinghua.edu.cn (T. H. Liu),    zou-wm@mail.tsinghua.edu.cn (W. M. Zou)} }
	
	\date{}
	\author{
		{\bf Tianhao Liu, \; Wenming Zou}\\
		\footnotesize \it  Department of Mathematical Sciences, Tsinghua University, Beijing 100084, China.\\
	}


	\maketitle
	
	\begin{center}
		\begin{minipage}{120mm}
			\begin{center}{\bf Abstract }\end{center}
			
	 In this paper, we	consider the logarithmic elliptic  equations with critical exponent
			\begin{equation}
				\begin{cases}
					-\Delta u=\lambda u+
					|u|^{2^*-2}u+\theta  u\log u^2, \\
					u \in H_0^1(\Omega), \quad  \Omega \subset \R^N.
				\end{cases}
			\end{equation}
			Here,  the parameters $N\geq 6$, $\lambda\in \R$, $\theta>0$ and  $ 2^*=\frac{2N}{N-2} $ is the Sobolev critical exponent.  We prove the existence of sign-changing solution with exactly two nodal domain for  an  arbitrary  smooth bounded domain $\Omega\subset \mathbb{R}^{N}$. When $\Omega=B_R(0)$ is a ball, we  also construct infinitely many radial sign-changing solutions with  alternating signs and prescribed nodal characteristic.

			\vskip0.23in
			

			\vskip0.23in					
		\end{minipage}
	\end{center}

	\section{Introduction}\label{Sect1}
	Consider the following  logarithmic elliptic  equations with critical exponent
	\begin{equation} \label{System1}
		\begin{cases}
			-\Delta u=\lambda u+
			|u|^{2^*-2}u+\theta  u\log u^2, \\
			u \in H_0^1(\Omega), \quad  \Omega \subset \R^N.
		\end{cases}
	\end{equation}
	Here, $\Omega\subset \mathbb{R}^{N}$ is a smooth bounded domain, $N\geq 3$, $\lambda\in \R$, $\theta>0$ and  $ 2^*=\frac{2N}{N-2} $ is the Sobolev critical exponent. Equation \eqref{System1} is closely related to the following time-dependent nonlinear logarithmic type Schr\"odinger  equation
	\begin{equation} \label{System5}
		\begin{cases}
			\imath \partial_t \Psi	=\Delta \Psi +
			|\Psi|^{2^*-2}\Psi+\theta \Psi\log |\Psi|^2 , \quad x\in \Omega,~ t>0,\\
			\Psi=\Psi(x,t)\in \C,\quad i=1,                                                                                                                    2\\
			\Psi(x,t)=0,\quad  x\in \partial \Omega,\quad t>0,
		\end{cases}
	\end{equation}
	where $\imath$ is the imaginary unit.  System \eqref{System5} appears in many physical fields, such as   quantum mechanics, quantum optics,
	nuclear physics, transport and diffusion phenomena,  open quantum systems, effective quantum gravity, theory
	of superfluidity and Bose-Einstein condensation. We  refer the readers to the papers \cite{Alfaro=DPDE=2017,Bialynicki=1975,Bialynicki=1976,Carles=2018,Carles=2014,Colin=2004,Poppenberg=2002}  for a survey on the related physical backgrounds.
	
	\vskip0.12in
	
	Equation \eqref{System1} can be also regarded as a logarithmic perturbation of  the classical 	Br\'ezis-Nirenberg  problem
	\begin{equation} \label{System3}
		-\Delta u=\lambda u+
		|u|^{2^*-2}u , ~~
		u \in H_0^1(\Omega), ~  \Omega \subset \R^N.
	\end{equation}
	Br\'ezis and Nirenberg, in a remarkable paper \cite{Brezis-Nirenberg1983},   proved that  the problem  \eqref{System3} has  a positive  solution if $ 0<\lambda<\lambda_{1}(\Omega) $ for $N\geq 4$ and $  \lambda^*(\Omega)<\lambda<  \lambda_{1}(\Omega)$ for $N=3$,  where $\lambda^*(\Omega)\in \sbr{0,\lambda_{1}(\Omega)}$. In \cite{Cerami-Solimini-Struwe 1986,Tarantello=1992,Zhang=NA=1989}, sign-changing solutions were obtained for $ 0<\lambda<\lambda_{1}(\Omega) $ and  $N\geq 6$. However, the situations are different for the lower dimensional problem. A well-known result  was presented in \cite{Atkinson-Brezis-Peletier}, demonstrating that when  $3\leq N\leq 6$ and $\Omega $ is  a unit ball, problem \eqref{System3} has no radial sign-changing  solution   for small   $0<\lambda<\lambda_{1}(\Omega)$. In \cite{Roselli-Willem 2009}, the authors proved that equation \eqref{System3} has a least energy sign-changing solution for $N=5$ and the parameter $\lambda$ is slightly smaller than $\lambda_{1}(\Omega)$.
	For more results related to the Br\'ezis-Nirenberg problem, we refer to \cite{Clapp-Weth 2005,Schechter-Zou}.

	\vskip0.12in

	Equation \eqref{System1} is associated with the energy functional $\mathcal{L}:H_0^1(\Omega) \to \R$ defined by
	
	\begin{equation} \label{Fun L}
		\begin{aligned}
			\mathcal{L}(u)&:=\frac{1}{2}\int_{\Omega} |\nabla u|^2-\frac{\lambda}{2}\int_{\Omega}|u|^2-\frac{1}{2^*} \int_{\Omega}|u|^{2^*}-\frac{\theta}{2}\int_{\Omega}u^2(\log u^2-1).
		\end{aligned}
	\end{equation}
	Obviously,  any   critical point of $\mathcal{L}$ corresponds to a solution of the equation \eqref{System1}. 	
	Recently, Deng et al.\cite{Deng-He-Pan=Arxiv=2022} investigated the existence of  positive  solutions for the equation  \eqref{System1} in a bounded domain via variational  method.
	They proved that the equation \eqref{System1} has  a positive  solution if $\lambda\in \R$, $\theta>0$ and $N\geq 4$. Moreover, they obtained the existence and nonexistence results under  other different conditions, we refer the readers to \cite{Deng-He-Pan=Arxiv=2022} for more details. 	
	In a particular case $\Omega=\R^N$, the existence of positive solutions and sign-changing solutions has been  studied in many situations,  we refer the readers to \cite{Cazenave=1983=NA,Deng=MAA=2021,D'Avenia=CCM=2014,W.Shuai=Nonlineariyu=2019,Squassina=2015=CVPDE,Wang=ARMA=2019,Zhang=JMPA=2020} and reference therein.

	

	\vskip0.12in
	The main focus of this paper is to  investigate the existence of sign-changing solutions for equation \eqref{System1} with logarithmic terms in a {\it bounded domain.}
	The logarithmic term   $u\log u^2$ has some special properties. It is easy to see that
	\begin{equation}
		\lim\limits_{u\to 0^+} \frac{u\log u^2}{u}=-\infty, \quad \lim\limits_{u\to +\infty} \frac{u\log u^2}{|u|^{2^*-2}u}=0,
	\end{equation}
	that is, $ u=o( u\log u^2)$ for $u$ very close to 0. Comparing with the critical term $|u|^{2^*-2}u$, the logarithmic term $u\log u^2$ is a lower-order term at infinity. Also, the sign of the logarithmic term changes depending on  $u$.  Moreover, the presence of logarithmic term  makes the structure of the corresponding functional complicated. 	In the following, we will demonstrate that   the  logarithmic terms in equation \eqref{System1}   presents some significant challenges and has a crucial influence on the existence of sign-changing solutions.
	
	Now we give some notations.	We define the Nehari set by
	\begin{equation}\label{Nehari manifold}
		\mathcal{N} :=\lbr{u\in  H_0^1(\Omega) :  u\not \equiv 0, ~ \mathcal{G}(u):=\mathcal{L}^\prime(u)u=0 },
	\end{equation}
	and the sign-changing  Nehari set by
	\begin{equation}\label{Nehari manifold signchaning}
		\mathcal{M}: =\lbr{u\in  \mathcal{N} :  u^+\in \mathcal{N} , u^-\in \mathcal{N}  },
	\end{equation}
	where   $u^+:=\max\lbr{u,0}$ and $u^-:=\min \lbr{u,0}$.
	It is easy to see that $\mathcal{N} $ and $	\mathcal{M} $ are nonempty, then we can denote
	\begin{equation} \label{Defi of energy }
		\mathcal{C} := \inf_{u\in \mathcal{N} } \mathcal{L}(u), \quad 	\mathcal{B} := \inf_{u\in \mathcal{M} } \mathcal{L}(u).
	\end{equation}
	Our main result is the following.
	\begin{theorem}\label{Theorem 1}
		Suppose that  $N\geq 6$ and $\lambda\in \R,\theta>0$, then $\mathcal{B} $ is  achieved. Moreover, 	if $u\in \mathcal{M} $ with $\mathcal{L}(u)=\mathcal{B} $, then $u$ is a sign-changing solution of equation \eqref{System1} with exactly two nodal domains.
	\end{theorem}
	\begin{remark}
		{\rm Comparing the results of \cite{Cerami-Solimini-Struwe 1986} and Theorem \ref{Theorem 1} above, we  can see  that equation \eqref{System1} has a sign-changing solution  for $0 < \lambda < \lambda_{1}(\Omega)$ and $N \geq 6$ if $\theta=0$, while it has a sign-changing solution for all $\lambda\in \R$ and $N\geq 6$ if $\theta>0$. So, the logarithmic term $\theta u\log u^2$ has much more influence than the term $\lambda u$ on the existence of  solutions.}
	\end{remark}
	To demonstrate the influence of logarithmic terms on the existence of sign-changing solutions, we will begin by reviewing the essential steps outlined in the work of Cerami, Solimini, and Struwe in 1986 \cite{Cerami-Solimini-Struwe 1986}, which studied the Br\'ezis-Nirenberg problem \eqref{System3}. They first obtained the existence of Palais-Smale sequence, which belongs to the corresponding sign-changing Nehari set, and then showed that the weak limit of such sequence is a  nontrivial  sign-changing solution   by establishing a key energy estimate.
	
	However,  when we try to apply their methods to the equation \eqref{System1}, we are faced with some difficulties.  One of the  main challenges arises from  the  logarithmic terms $u\log u^2$, which  makes it complicated to establish the corresponding energy estimate. This requires us to develop new ideas and perform additional calculations, we refer to  Section \ref{Sect2} for details. Furthermore, because of the sign of $u\log u^2$ is indefinite,   the approach  in \cite{Cerami-Solimini-Struwe 1986} for  obtaining a  Palais-Smale sequence belonging to $\mathcal{M}$ cannot be applied
	to our case.  In order to overcome this difficulty, we adopt some ideas in the classical work  \cite{Devillanova=ADE=2002} by using the subcritical approximation method.
	  More precisely, we first establish the existence of sign-changing solutions for the subcritical problem
	\begin{equation}
		-\Delta u=\lambda u+
		|u|^{p-2}u+\theta  u\log u^2, ~~
		u \in H_0^1(\Omega),
	\end{equation}
	with $2<p<2^*$. Then we study the behaviors of these solutions and finally take the limit as the exponent $p$ approaches  $2^*$ to obtain a sign-changing  solution of the critical problem \eqref{System1}.
	
	\medbreak
	In a particular case $\Omega=B_R:=\lbr{ x\in \R^N:|x|\leq R}$, we can construct infinitely many radial sign-changing solutions for equation \eqref{System1}. Then the existence result   in this aspect  can be stated as follows.
	\begin{theorem}\label{Theorem 2}
		Suppose that $\Omega=B_R $, $N\geq 6$ and $\lambda\in \R,\theta>0$. Then for any $k\in\N$, there exists a pair of radial solutions $u_{+}$ and $u_{-}$, $u_{\pm}(x)=u_{\pm}\sbr{|x|}$ of equation \eqref{System1} with the following properties:
		\begin{enumerate}
			\item[(i)] $u_{-}(0)<0<u_{+}(0)$,
			\item[(ii)]    $u_{\pm}$ possess exactly $k$ nodes with $0<r_1<r_2<\ldots<r_k:=R$ and $u_{+}(r_i)=u_{-}(r_i)=0$ for $i=1,2,\ldots,k$.
		\end{enumerate}
	\end{theorem}
	
	
	\begin{remark}
		{\rm According to \cite{Cerami-Solimini-Struwe 1986}, the infinitely many radial sign-changing solutions of  Br\'ezis-Nirenberg problem \eqref{System3} can be obtained only for $N\geq 7$ and  $ \lambda \in \sbr{0,\lambda_{1}(\Omega)}$. However, the presence of logarithmic terms in equation \eqref{System1}  allows for the existence of infinitely many radial sign-changing solutions for $N\geq 6$ and $\lambda\in \R$.}
	\end{remark}

	%
	%
	%
	
	\medbreak
	Throughout this paper, we denote the norm of $L^p(\Omega)$  by $|\cdot|_p$ for  $1\leq p \leq \infty$.  We use ``$\to$'' and ``$\rightharpoonup$'' to denote the strong convergence and weak convergence in corresponding space respectively. The capital letter $C$ will appear as a constant which may vary from line to line.
	
	This paper is organized as follows. In Section \ref{Sect2}, we will introduce an important energy estimate, which is used to prove Theorem \ref{Theorem 1}, and complete its proof in Section \ref{Sect3}. In Section \ref{Sect4}, we will construct infinitely many radial sign-changing solutions for equation \eqref{System1} and prove Theorem \ref{Theorem 2}.
	
	\section{Energy estimates }\label{Sect2}
	In this section, we will  give an  energy estimate that $\mathcal{B} <\mathcal{C} +\frac{1}{N} \mathcal{S}^{\frac{N}{2}}$, which is crucially important to the proof of Theorem \ref{Theorem 1}. Inspired by Cerami-Solimini-Struwe \cite{Cerami-Solimini-Struwe 1986}, we only need to find some appropriate $\psi_\varepsilon \in H_0^1(\Omega)$ such that  $\sup_{\alpha,\beta\in \R} \mathcal{L}(\alpha u_g +\beta \psi_\varepsilon )<\mathcal{C} +\frac{1}{N} \mathcal{S}^{\frac{N}{2}} $, where $u_g$ is a positive ground state solution of equation \eqref{System1}. However,   the  logarithmic  nonlinearity causes some obstacles, so we can not use the methods of   \cite{Cerami-Solimini-Struwe 1986} directly, and  we require some  new  ideas.
	
	Without loss of generality, we may assume that $B_{2\rho}(0)\subset \Omega$   for some appropriate $\rho >0$.
	For $\varepsilon>0$ and $y \in \R^N$, we consider the Aubin-Talenti bubble (\cite{Aubin=JDG=1976},\cite{Talenti=AMPA=1976})  $U_{\varepsilon,y}\in\mathcal{D}^{1,2}(\R^N) $ defined by
	\begin{equation} \label{limit system_1}
		U_{\varepsilon,y}(x)= \cfrac{\mbr{N(N-2))\varepsilon^2}^{\frac{N-2}{4}}}{\sbr{\varepsilon^2+|x-y|^2}^{\frac{N-2}{2}}}~.
	\end{equation}
	Then $U_{\varepsilon,y}$ solves the equation
	\begin{equation}
		-\Delta u = u^{\frac{N+2}{N-2}} \text{ in } \R^N,
	\end{equation}
	and
	\begin{equation}
		\int_{\R^3}|\nabla U_{\varepsilon,y}|^2  =\int_{\R^3}| U_{\varepsilon,y}|^{2^*} = \mathcal{S}^{\frac{N}{2}}.
	\end{equation}
	Here, $\mathcal{S}$ is the Sobolev best constant of $\mathcal{D}^{1,2}(\R^N)\hookrightarrow L^{2^*}(\R^N)$,
	\begin{equation}\label{sobolev constant}
		\mathcal{S}=\inf_{u \in \mathcal{D}^{1,2}(\R^N) \setminus \lbr{0} } \cfrac{\int_{\R^N } |\nabla u|^2 }{\sbr{\int_{\R^N} |u|^{2^*}}^{\frac{2}{2^*}}}~,
	\end{equation}
	where $\mathcal{D}^{1,2}(\R^N)=\lbr{u\in L^2(\R^N): |\nabla u| \in L^2(\R^N)}$ with norm $\left\|u \right\|_{\mathcal{D}^{1,2}}:=\sbr{\int_{\R^N}|\nabla u|^2 }^{\frac{1}{2}} $.  To simplify the notation, we denote
	\begin{equation} \label{defi of U{varepsilon}}
		U_{\varepsilon}(x):=U_{\varepsilon,0}(x)= \cfrac{\mbr{N(N-2))\varepsilon^2}^{\frac{N-2}{4}}}{\sbr{\varepsilon^2+|x|^2}^{\frac{N-2}{2}}}~.
	\end{equation}

	Let $\xi\in C_0^\infty (\Omega)$ be the radial  function,
	such that  $\xi(x)\equiv1$ for $0\leq |x|\leq \rho$, $0\leq \xi(x)\leq 1$ for $\rho \leq |x|\leq 2\rho $, and $\xi(x)\equiv0$ for $ |x|\geq  2\rho$, and we denote
	\begin{equation}
		\psi_\varepsilon(x):=\xi(x)U_{\varepsilon}(x).
	\end{equation}
	Then we have the following well-known energy estimates.
	
	\begin{lemma} \label{Lemma ES-1} Let $N\geq 5$, then we have, as $\varepsilon\to 0^+$,
		\begin{equation} \label{es-1}
			\int_{\Omega}|\nabla \psi_\varepsilon|^2=\mathcal{S}^{\frac{N}{2}}+O(\varepsilon^{N-2}),
		\end{equation}
		\begin{equation}\label{es-2}
			\int_{\Omega} |\psi_\varepsilon|^{2^*}  = \mathcal{S}^{\frac{N}{2}}+ O(\varepsilon^N),
		\end{equation}
		\begin{equation} \label{es-3}
			\int_{\Omega} |\psi_\varepsilon|^2 =
			C_0\varepsilon^2+O(\varepsilon^{N-2}),
		\end{equation}
		and
		\begin{equation}\label{es-4}
			\int_{\Omega} |\psi_\varepsilon| =O(\varepsilon^{\frac{N-2}{2}}), \quad \int_{\Omega} |\psi_\varepsilon|^{2^*-1} =O(\varepsilon^{\frac{N-2}{2}}),
		\end{equation}
		where $	C_0$ is a positive constant.
		
	\end{lemma}
	\begin{proof}
		The proofs  can be found in \cite[Lemma 2.1]{Capozzi=poincare1985} and  \cite[Lemma 1.46]{william=1996}.
	\end{proof}
	
	
	\begin{lemma}\label{Lemma ES-3}
		Let $N\geq 5$, then we have, as $\varepsilon\to 0^+$,
		\begin{equation}\label{es-7}
			\int_{\Omega} \psi_\varepsilon^2 \log \psi_\varepsilon^2 =C_1 \varepsilon^2|\log \varepsilon|+O(\varepsilon^2),
		\end{equation}
		where $C_1$ is a positive constant.
	\end{lemma}
	\begin{proof}
		The proofs  can be found in  \cite[Lemma 3.2]{Deng-He-Pan=Arxiv=2022}, so we omit it.
	\end{proof}
	%
	%
	%
	
	Assume that $N\geq 4$, $\lambda\in \R$ and $\theta>0$, by \cite{Deng-He-Pan=Arxiv=2022}    the Br\'ezis-Nirenberg problem with logarithmic perturbation
	\begin{equation}
		-\Delta u=\lambda u+
		|u|^{2^*-2}u+\theta  u\log u^2, \ \
		u \in H_0^1(\Omega),
	\end{equation}
	has a positive ground state solution $u_g \in C^2(\Omega)$ with $\mathcal{L}(u_g)=\mathcal{C} >0$ and
	\begin{equation}
		\int_{\Omega}|\nabla u_g|^2=\lambda \int_{\Omega} |u_g|^2+\int_{\Omega} |u_g|^{2^*}+\theta \int_{\Omega}u_g^2\log u_g^2.
	\end{equation}
	Moreover, there exists a constant $C>0$ such that $|u_g|_{L^{\infty}(\Omega)} \leq C$. Then we have the following result.
	
	%
	
	%
	%
	

	\begin{lemma} \label{Lemma ES-4} Assume that  $N\geq 4$ and $\alpha>0,\beta<0$ are bounded, then we have, as $\varepsilon\to 0^+$,
		\begin{equation}\label{es-5}
			\int_{\Omega} |\alpha u_g +\beta \psi_\varepsilon|^{2^*} =	\int_{\Omega} |\alpha u_g |^{2^*}+	\int_{\Omega} | \beta \psi_\varepsilon|^{2^*}+O(\varepsilon^{\frac{N-2}{2}}),
		\end{equation}
		and
		\begin{equation}\label{es-6}
			\begin{aligned}
				&\int_{\Omega}\sbr{\alpha u_g +\beta \psi_\varepsilon}^2\log \sbr{\alpha u_g +\beta \psi_\varepsilon}^2\geq \int_{\Omega}\sbr{\alpha u_g}^2\log \sbr{\alpha u_g }^2+\int_{\Omega}\sbr{\beta \psi_\varepsilon}^2\log \sbr{\beta \psi_\varepsilon}^2+O(\varepsilon^{\frac{N-2}{2}}).
			\end{aligned}
		\end{equation}
	\end{lemma}

	\begin{proof}
		The proof of \eqref{es-5} follows directly from the strategies in \cite[Formula 2.11]{Cerami=Ponicare1984},   we omit the  details here. Now we prove the second estimate \eqref{es-6}.
		Notice that
		\begin{equation}
			\begin{aligned}
				&\int_{\Omega}\sbr{\alpha u_g +\beta \psi_\varepsilon}^2\log \sbr{\alpha u_g +\beta \psi_\varepsilon}^2-\int_{\Omega}\sbr{\alpha u_g}^2\log \sbr{\alpha u_g }^2-\int_{\Omega}\sbr{\beta \psi_\varepsilon}^2\log \sbr{\beta \psi_\varepsilon}^2 \\
				& =\int_{\Omega}   \alpha u_g  \mbr{\sbr{\alpha u_g +\beta \psi_\varepsilon}\log \sbr{\alpha u_g +\beta \psi_\varepsilon}^2-\sbr{\alpha u_g }\log \sbr{\alpha u_g }^2}
				\\&	+\int_{\Omega}\beta \psi_\varepsilon \mbr{\sbr{\alpha u_g +\beta \psi_\varepsilon}\log \sbr{\alpha u_g +\beta \psi_\varepsilon}^2-\sbr{\beta \psi_\varepsilon}\log \sbr{\beta \psi_\varepsilon}^2} \\
				&=: I_1+I_2 .
			\end{aligned}
		\end{equation}
		By a direct calculation,  we have
		\begin{equation}
			\begin{aligned}
				I_1&= \int_{\Omega}   \alpha u_g  \int_{0}^{1}\frac{\partial }{\partial \sigma} \mbr{\sbr{\alpha u_g + \sigma\beta \psi_\varepsilon}\log \sbr{\alpha u_g + \sigma\beta \psi_\varepsilon}^2}d \sigma d x\\
				&= \int_{\Omega}\int_{0}^{1}  \alpha \beta u_g \psi_\varepsilon\log \sbr{\alpha u_g + \sigma\beta \psi_\varepsilon}^2 d \sigma d x
				+ \int_{\Omega}\int_{0}^{1}  2\alpha \beta u_g\psi_\varepsilon        d \sigma d x \\
				&=:I_{11}+I_{12}.
			\end{aligned}
		\end{equation}
		By Lemma \ref{es-1},we have
		\begin{equation}
			\begin{aligned}
				I_{12}\leq C|u_g|_{L^{\infty}(\Omega)}|\psi_\varepsilon|_1^1=O(\varepsilon^{\frac{N-2}{2}}).
			\end{aligned}
		\end{equation}
		Since $\alpha>0 $ and $\beta<0$, we deduce that
		\begin{equation}
			\begin{aligned}
				I_{11}&=\int_{0}^{1} \int_{\Omega_{1,\sigma}}  \alpha \beta u_g \psi_\varepsilon\log \sbr{\alpha u_g + \sigma\beta \psi_\varepsilon}^2  d x d \sigma +\int_{0}^{1} \int_{\Omega_{2 ,\sigma}}  \alpha \beta u_g \psi_\varepsilon\log \sbr{\alpha u_g + \sigma\beta \psi_\varepsilon}^2 d x d \sigma \\
				&\geq \int_{0}^{1} \int_{\Omega_{1,\sigma}}  \alpha \beta u_g \psi_\varepsilon\log \sbr{\alpha u_g + \sigma\beta \psi_\varepsilon}^2 d x d \sigma ,
			\end{aligned}
		\end{equation}
		where
		\begin{equation}
			\Omega_{1,\sigma}:=\lbr{x\in \Omega: | \alpha u_g(x) + \sigma\beta \psi_\varepsilon(x)| >1},
		\end{equation}
		\begin{equation}
			\Omega_{2,\sigma}:=\lbr{x\in \Omega: | \alpha u_g(x) + \sigma\beta \psi_\varepsilon(x)| \leq 1}.
		\end{equation}
		Note that  $\log s^2 \leq C s^{2^*-2}$ for  $s>1$ and some $C>0$, then by using Lemma  \ref{es-1} again, we have
		\begin{equation}
			\begin{aligned}
				\left| \int_{0}^{1} \int_{\Omega_{1,\sigma}}  \alpha \beta u_g \psi_\varepsilon\log \sbr{\alpha u_g + \sigma\beta \psi_\varepsilon}^2 d x d \sigma \right| &\leq C \int_{0}^{1} \int_{\Omega_{1,\sigma}} | u_g| | \psi_\varepsilon| |\alpha u_g + \sigma\beta \psi_\varepsilon|^{2^*-2}d x d \sigma \\
				&\leq C\int_{0}^{1} \int_{\Omega_{1,\sigma}} \sbr{|u_g|^{2^*-1}|\psi_\varepsilon|+\sigma^{2^*-2}|u_g|| \psi_\varepsilon|^{2^*-1}}d x d \sigma\\
				&\leq C\sbr{|u_g|_{L^{\infty}(\Omega)}|\psi_\varepsilon|_{2^*-1}^{2^*-1}+|u_g|_{L^{\infty}(\Omega)}^{2^*-1} |\psi_\varepsilon|_1 }= O(\varepsilon^{\frac{N-2}{2}}).
			\end{aligned}
		\end{equation}
		Therefore, $	I_1\geq O(\varepsilon^{\frac{N-2}{2}})$. Similarly, we can prove that $I_2\geq O(\varepsilon^{\frac{N-2}{2}})$. Then \eqref{es-6} holds. This completes the proof.
		
		
	\end{proof}

	\begin{proposition}\label{Lemma energy estimate1}
		Suppose that $N\geq 6$,  $\lambda\in \R$ and $\theta>0$, then we have
		\begin{equation}
			\mathcal{B} <\mathcal{C} +\frac{1}{N} \mathcal{S}^{\frac{N}{2}}.
		\end{equation}
	\end{proposition}
	
	\begin{proof}
		Obviously,  we have
		\begin{equation} \label{s1}
			\begin{aligned}
				\mathcal{L}(\alpha u_g +\beta \psi_\varepsilon )&=\frac{\alpha^2}{2}\int_{\Omega}\sbr{|\nabla u_g|^2-\lambda|u_g|^2}+\frac{\alpha^2}{2}\theta\int_{\Omega}|u_g|^2+\frac{\beta^2}{2}\int_{\Omega}\sbr{|\nabla \psi_\varepsilon|^2-\lambda|\psi_\varepsilon|^2}+\frac{\beta^2}{2}\theta\int_{\Omega}|\psi_\varepsilon|^2\\
				&+\alpha\beta\int_{\Omega} \sbr{\nabla u_g\nabla\psi_\varepsilon-\lambda u_g\psi_\varepsilon}+\alpha\beta  \int_{\Omega}\theta u_g\psi_\varepsilon-\frac{1}{2^*}\int_{\Omega}|\alpha u_g +\beta \psi_\varepsilon|^{2^*}\\
				&-\frac{\theta}{2}\int_{\Omega}\sbr{\alpha u_g +\beta \psi_\varepsilon}^2\log \sbr{\alpha u_g +\beta \psi_\varepsilon}^2.
			\end{aligned}
		\end{equation}
		Then it is sufficient to show that
		\begin{equation}
			\sup_{\alpha,-\beta\in \R^+} \mathcal{L}(\alpha u_g +\beta \psi_\varepsilon )<\mathcal{C} +\frac{1}{N} \mathcal{S}^{\frac{N}{2}}.
		\end{equation}
		As a matter of fact, by the Miranda's theorem (see \cite{Miranda}), there exist $\tilde{\alpha}>0,\tilde{\beta}<0$ such that $\tilde{\alpha}u_g+\tilde{\beta}\psi_\varepsilon \in \mathcal{M} $, then we have $$\mathcal{B} \leq \mathcal{L}(\tilde{\alpha}u_g+\tilde{\beta}\psi_\varepsilon )\leq\sup_{\alpha,-\beta\in \R^+} \mathcal{L}(\alpha u_g +\beta \psi_\varepsilon ) <\mathcal{C} +\frac{1}{N} \mathcal{S}^{\frac{N}{2}}.$$
		
		Now we claim that there exists $r_0>0$, such that for all $\alpha^2+\beta^2=r ^2\geq r_0^2$, there holds
		\begin{equation}
			\mathcal{L}(\alpha u_g +\beta \psi_\varepsilon ) \leq 0.
		\end{equation}
		For that purpose, we set $\bar{\alpha}=\alpha/r$, $\bar{\beta}=\beta/r$, and $\phi =\bar{\alpha}u_g +\bar{\beta} \psi_\varepsilon $, then $\bar{\alpha}^2+\bar{\beta}^2=1$ and $\int_{\Omega}|\nabla \phi|^2, \int_{\Omega}|\phi|^2$,  $\int_{\Omega}|\phi|^4$, and $\int_{\Omega}\phi^2 \log \phi^2$ are all bounded from below and above. Therefore, the claim follows directly from
		\begin{equation}
			\begin{aligned}
				\mathcal{L}(r\phi)&= \frac{r^2}{2}\int_{\Omega} \mbr{|\nabla \phi|^2+\sbr{\theta-\lambda}|\phi|^2-\theta \phi^2\log \phi^2}-\frac{r^{2^*}}{2^*} \int_{\Omega}|\phi|^{2^*}-\frac{r^2\log r^2}{2}\theta\int_{\Omega}\phi^2 \\
				&\to -\infty, \quad \text{ as } r\to +\infty.
			\end{aligned}
		\end{equation}
		Based on this claim,  we can assume that $\alpha>0$ and $\beta<0$ are bounded. With this fact in mind, we can deduce from  Lemma \ref{es-1} and $u_g\in C^2(\Omega)$ that
		\begin{equation}
			\alpha\beta\int_{\Omega} \sbr{\nabla u_g\nabla\psi_\varepsilon-\lambda u_g\psi_\varepsilon}+\alpha\beta  \int_{\Omega}\theta u_g\psi_\varepsilon  = O(\varepsilon^{\frac{N-2}{2}}).
		\end{equation}
		Combining this with Lemma \ref{Lemma ES-4}, we  can see from \eqref{s1} that
		\begin{equation}
			\begin{aligned}
				\mathcal{L}(\alpha u_g +\beta \psi_\varepsilon )&\leq \frac{\alpha^2}{2}\int_{\Omega}\sbr{|\nabla u_g|^2-\lambda|u_g|^2}+\frac{\alpha^2}{2}\theta\int_{\Omega}|u_g|^2-\frac{\alpha^{2^*}}{2^*}\int_{\Omega}| u_g|^{2^*}-\frac{\theta}{2}\int_{\Omega}\sbr{\alpha u_g}^2\log \sbr{\alpha u_g }^2\\
				&+\frac{|\beta|^2}{2}\int_{\Omega}\sbr{|\nabla \psi_\varepsilon|^2-\lambda|\psi_\varepsilon|^2}+\frac{|\beta|^2}{2}\theta\int_{\Omega}|\psi_\varepsilon|^2-\frac{|\beta|^{2^*}}{2^*}\int_{\Omega}|   \psi_\varepsilon|^{2^*}-\frac{\theta}{2}\int_{\Omega}\sbr{\beta \psi_\varepsilon}^2\log \sbr{\beta \psi_\varepsilon}^2\\&+O(\varepsilon^{\frac{N-2}{2}})=:f_1(\alpha)+f_2(\beta)+O(\varepsilon^{\frac{N-2}{2}}).
			\end{aligned}
		\end{equation}
		Recalling that $u_{g}$ is the positive least energy solution of $-\Delta u=\lambda u+ |u|^{2^*-2}u+ \theta u\log u^2$. Then, we have
		\begin{equation}\label{f3}
			\int_{\Omega}|\nabla u_{g}|^2=\lambda \int_{\Omega}|u_{g}|^2+ \int_{\Omega}|u_{g}|^{2^*}+\theta  \int_{\Omega}u_{g}^2\log u_{g}^2,
		\end{equation}
		and
		\begin{equation}\label{f31}
			\begin{aligned}
				\mathcal{C}_{g}&=\frac{1}{2}\int_{\Omega}|\nabla u_{g}|^2-\frac{\lambda }{2}\int_{\Omega}|u_{g}|^2-\frac{ 1}{2^*}\int_{\Omega}|u_{g}|^{2^*} -\frac{\theta }{2}\int_{\Omega}u_{g}^2(\log u_{g}^2-1).
			\end{aligned}
		\end{equation}
		By a direct calculation, we can see from \eqref{f3} that
		\begin{equation}\label{F8}
			\begin{aligned}
				f^\prime(\alpha)&=\alpha\int_{\Omega}|\nabla u_{g}|^2-\alpha\lambda\int_{\Omega}|u_{g}|^2-\alpha^{2^*-1}  \int_{\Omega}|u_{g}|^{2^*} -\theta\int_{\Omega}\alpha u_{g}\log (\alpha u_{g})^2\\
				&=(\alpha-\alpha^{2^*-1})\int_{\Omega}|u_{g}|^{2^*} -( \alpha\log\alpha^2)\theta \int_{\Omega}|u_{g}|^2.
			\end{aligned}
		\end{equation}
		Since $\theta>0$ and $2^*>2$, one can easily check that $f^\prime(\alpha)> 0$ for $0<\alpha< 1$ and $f^\prime(\alpha)< 0$ for $\alpha> 1$. So,  by \eqref{f31},
		\begin{equation}\label{s3}
			f_1(\alpha)\leq f_1(1)=\mathcal{C}_{g}.
		\end{equation}
		On the other hand, since $|\beta|$ is bounded, it follows from Lemma \ref{Lemma ES-1} and \ref{Lemma ES-3} that
		\begin{equation}
			\begin{aligned}
				f_2(\beta)&=\sbr{\frac{|\beta|^2}{2}\int_{\Omega}|\nabla \psi_\varepsilon|^2-\frac{|\beta|^{2^*}}{2^*}\int_{\Omega}|   \psi_\varepsilon|^{2^*}}+\frac{\theta-\lambda}{2}|\beta|^2\int_{\Omega}|\psi_\varepsilon|^2-\frac{\theta}{2}|\beta|^2\int_{\Omega}\psi_\varepsilon\log  \psi_\varepsilon^2
				-\frac{\theta}{2}|\beta|^2\log |\beta|^2 \int_{\Omega}|\psi_\varepsilon|^2\\
				&\leq \sbr{\frac{1}{2}|\beta|^2-\frac{1}{2^*}|\beta|^{2^*}}\mathcal{S}^{\frac{N}{2}}+O(\varepsilon^{N-2}) -C \theta \varepsilon^2|\log \varepsilon|+O(\varepsilon^2)\\
				&\leq \frac{1}{N}\mathcal{S}^{\frac{N}{2}} -C \theta \varepsilon^2|\log \varepsilon|+O(\varepsilon^2).
			\end{aligned}
		\end{equation}
		Since $N\geq 6$ and $\theta>0$, we have
		\begin{equation}
			\begin{aligned}
				\mathcal{L}(\alpha u_g +\beta \psi_\varepsilon ) &\leq \mathcal{C}_{g}+ \frac{1}{N}\mathcal{S}^{\frac{N}{2}} -C \theta \varepsilon^2|\log \varepsilon|+O(\varepsilon^2)+O(\varepsilon^{\frac{N-2}{2}})\\
				& <\mathcal{C}_{g}+ \frac{1}{N}\mathcal{S}^{\frac{N}{2}},
			\end{aligned}
		\end{equation}
		for $\varepsilon$ small enough. This completes the proof.
	\end{proof}
	\vskip 0.1in
	
	\section{ Proof of Theorem \ref{Theorem 1}}\label{Sect3}
	In this section, we will prove Theorem \ref{Theorem 1}  by using the subcritical approximation method. Now we consider the following subcritical problem
	
	\begin{equation} \label{System2}
		\begin{cases}
			-\Delta u=\lambda u+
			|u|^{p-2}u+\theta  u\log u^2, \\
			u \in H_0^1(\Omega), \quad  \Omega \subset \R^N,
		\end{cases}
	\end{equation}
	where  $N\geq 4$, $\lambda\in \R$, $\theta>0$ and   $p\in \sbr{2,2^*}$. Define the associated functional by
	\begin{equation} \label{Fun L_p}
		\begin{aligned}
			\mathcal{L}_p(u)&:=\frac{1}{2}\int_{\Omega} |\nabla u|^2-\frac{\lambda}{2}\int_{\Omega}|u|^2-\frac{1}{p} \int_{\Omega}|u|^{p}-\frac{\theta}{2}\int_{\Omega}u^2(\log u^2-1),
		\end{aligned}
	\end{equation}
	and the energy level by
	\begin{equation} \label{Defi of energy-p}
		\mathcal{B}_p := \inf_{u\in \mathcal{M}_p} \mathcal{L}_p(u),
	\end{equation}
	where
	\begin{equation}\label{Nehari manifold N_p}
		\mathcal{N}_{p}:=\lbr{u\in  H_0^1(\Omega) :  u\not \equiv 0,  \mathcal{G}_p(u):=\mathcal{L}_p^\prime(u)u=0 },
	\end{equation}
	and
	\begin{equation}\label{Nehari manifold signchaning N_{sp}}
		\mathcal{M}_{p}:=\lbr{u\in  \mathcal{N}_p:  u^+\in \mathcal{N}_p, u^-\in \mathcal{N}_p }.
	\end{equation}
	Then we  have  the following result.
	\begin{proposition}\label{Lemma achieved p}
		$\mathcal{B}_p$ is  achieved. Moreover,  if  $u \in \mathcal{M}_p$ such that $\mathcal{L}_p(u)=\mathcal{B}_p$, then $u$ is a  solution of the subcritical problem \eqref{System2}.
	\end{proposition}
	
	\begin{proof}
		The proof is inspired by \cite{W.Shuai=Nonlineariyu=2019}, we give the details here for the reader's convenience.
		Take $\lbr{u_n} \subset \mathcal{M}_p$ be a minimizing sequence of $\mathcal{B}_p$, that is,
		\begin{equation}\label{s2}
			\mathcal{B}_p=\lim_{n\to \infty}	\mathcal{L}_p(u_n)=\lim_{n\to \infty}	\mbr{\mathcal{L}_p(u_n)-\frac{1}{2}\mathcal{G}_p(u_n)}=\lim_{n\to \infty} \mbr{\sbr{\frac{1}{2}-\frac{1}{p}}\int_{\Omega}|u_n|^{p}+\frac{\theta}{2}\int_{\Omega}|u_n|^{2}}.
		\end{equation}
		Then, $\lbr{u_n}$ is bounded in $L^2(\Omega)$ and $L^{p}(\Omega)$.
		
		Recalling the following useful inequality (see  \cite[Theorem 8.14]{Lieb=2001})
		\begin{equation}
			\int_{\Omega} u^2 \log u^2 \leq \frac{a}{\pi}|\nabla u|_2^2+\sbr{\log |u|_2^2-N(1+\log a)}|u|_2^2  \ \ \text{ for } u\in H_0^1(\Omega) \  \text{ and } \  a>0.
		\end{equation}
		Combining this inequality with the fact that $u_n\in \mathcal{M}_p$, one can easily see that $\lbr{u_n}$ is bounded in $H_0^1(\Omega)$ by taking   $2a\leq \pi$. Hence, passing to subsequence, we may assume that there exists $u\in H_0^1(\Omega)$ such that
		\begin{equation}\label{s25}
			\begin{aligned}
				&	u_n \rightharpoonup u \text{ weakly in } H_0^1(\Omega),\\
				&u_n \to u \ \text{ strongly in } L^{p}(\Omega) \text{ for } 2\leq p<2^*,\\
				&u_n \to u \ \text{ almost everywhere in } \Omega.
			\end{aligned}
		\end{equation}
		Since $\lbr{u_n} \subset \mathcal{M}_p$ and $s^2\log s^2\leq C_p s^{p}$ for $s>0$ and some constant $C_p>0$, we have the following
		\begin{equation}
			\begin{aligned}
				C|u_n^+|_{p}^{\frac{2}{p}}	\leq |\nabla u_n^+|_2^2
				=   |u_n^+|_{p}^{p} + \theta \int_{\Omega}( u_n^+)^2\log (e^{\frac{\lambda}{\theta}}(u_n^+)^2)
				\leq  C_p|u_n^+|_{p}^{p},
			\end{aligned}
		\end{equation}
		\begin{equation}
			\begin{aligned}
				C|u_n^-|_{p}^{\frac{2}{p}}	\leq |\nabla u^-|_2^2
				=   |u_n^-|_{p}^{p} + \theta \int_{\Omega}( u_n^-)^2\log (e^{\frac{\lambda}{\theta}}(u_n^-)^2)
				\leq  C_p|u_n^-|_{p}^{p},
			\end{aligned}
		\end{equation}
		which implies that $|u_n^+|_{p}^{p}\geq  C>0$ and $|u_n^-|_{p}^{p}\geq  C>0$. Then we can see from  \eqref{s25} that
		\begin{equation}
			|u ^+|_{p}^{p}\geq  C>0 ~\text{ and }~ |u ^-|_{p}^{p}\geq  C>0,
		\end{equation}
		which yields that $u^+\neq 0$ and $u^-\neq 0$.
		By using the weak lower semi-continuity  of norm and   the Sobolev compact embedding theorem again, we have
		\begin{equation}
			\begin{aligned}
				\int_{\Omega} |\nabla u^+|^2
				&\leq \liminf_{n\to \infty}		 \int_{\Omega}  |\nabla u_n^+|^2=\liminf_{n\to \infty}	 \mbr{\lambda\int_{\Omega}  |u_n^+|^2+\int_{\Omega} |u_n^+|^{p}+\theta\int_{\Omega}(u_n^+)^2\log (u_n^+)^2}\\
				&=\lambda \int_{\Omega} |u^+|^2+\int_{\Omega} |u^+|^{p}+\theta\int_{\Omega}(u^+)^2\log (u^+)^2,
			\end{aligned}-
		\end{equation}
		where the last equality is guaranteed by \cite[Lemma 2.3]{Deng-He-Pan=Arxiv=2022}.  Similarly, we have
		\begin{equation}
			\int_{\Omega} |\nabla u^-|^2 \leq \lambda\int_{\Omega}  |u^-|^2+\int_{\Omega} |u^-|^{p}+\theta\int_{\Omega}(u^-)^2\log (u^-)^2.
		\end{equation}
		Thus, there exist $s,t\in \left( 0,1\right] $ such that $\bar{u}=su^++tu^-\in \mathcal{M}_p$. Since $p\in \sbr{2,2^*}$, we deduce from \eqref{s2} that
		\begin{equation}
			\begin{aligned}
				\mathcal{B}_p&\leq \mathcal{L}_p(\bar{u})=\sbr{\frac{1}{2}-\frac{1}{p}}\sbr{s^{p}\int_{\Omega}|u^+ |^{p}+t^{p}\int_{\Omega}|u^- |^{p}}+\frac{\theta}{2}\sbr{s^2\int_{\Omega}|u^+|^{2}+t^2\int_{\Omega}|u^-|^{2}}\\
				&\leq \sbr{\frac{1}{2}-\frac{1}{p}}\int_{\Omega}|u |^{p}+\frac{\theta}{2}\int_{\Omega}|u|^{2} 	\leq \liminf_{n\to \infty}\mbr{\sbr{\frac{1}{2}-\frac{1}{p}}\int_{\Omega}|u_n |^{p}+\frac{\theta}{2}\int_{\Omega}|u_n|^{2} }\\
				& =	\mathcal{B}_p.
			\end{aligned}
		\end{equation}
		Therefore,  we conclude that $s=t=1$  and  $\mathcal{B}_p$ is achieved by $\bar{u} \in \mathcal{M}_p$, that is,    $\mathcal{L}_p(\bar{u})=\mathcal{B}_p$.
		
		Now we prove that $\bar{u}$ is a weak solution of the subcritical problem \eqref{System2}, that is, $\mathcal{L}_p^\prime(\bar{u})=0$.  Suppose by contrary that   the conclusion  does not hold. Then there exists a function $\varphi\in C_0^\infty(\Omega)$ such that
		\begin{equation}
			\mathcal{L}_p^\prime (\bar{u} )\varphi \leq -1.
		\end{equation}
		Taking $\varepsilon>0$ small enough such that
		\begin{equation} \label{f32}
			\mathcal{L}_p^\prime (s\bar{u} ^++t\bar{u} ^++\gamma\varphi )\varphi \leq -\frac{1}{2}, \text{ for all } |s-1|+|t-1|+|\gamma|\leq 3\varepsilon.
		\end{equation}
		Let $\eta: \R^2 \to \mbr{0,1}$ be a cut-off function such that
		\begin{equation} \label{f34}
			\eta(s,t)=\begin{cases}
				1, \ \ \text{ if } |s-1|\leq \frac{1}{2}\varepsilon \text{ and } |t-1|\leq \frac{1}{2}\varepsilon,\\
				0 ,  \ \ \text{ if } |s-1|\geq  \varepsilon \text{ or  } |t-1|\geq  \varepsilon,\\
				\in (0,1), \ \ \text{ otherwise}.
			\end{cases}
		\end{equation}
		Then if $|s-1|\leq \varepsilon$ and $|t-1|\leq \varepsilon$, we can see from \eqref{f32} that
		\begin{equation}\label{f33}
			\begin{aligned}
				\mathcal{L}_p\sbr{s\bar{u}^++t\bar{u}^-+\varepsilon\eta(s,t)\varphi}
				&=\mathcal{L}_p\sbr{s\bar{u}^++t\bar{u}^-}+\sigma\varepsilon\eta(s,t)\int_{0}^{1}\mathcal{L}_p^\prime \sbr{s\bar{u}^++t\bar{u}^++\sigma\varepsilon\eta(s,t) }\varphi d \sigma\\
				&\leq \mathcal{L}_p\sbr{s\bar{u}^++t\bar{u}^-}-\frac{1}{2}\sigma\varepsilon\eta(s,t).
			\end{aligned}
		\end{equation}	
		If $|s-1|\leq \varepsilon$ or  $|t-1|\leq \varepsilon$, then $\eta(s,t)=0$ and the above inequality holds true.
		Observe that $$f(s,t):=\mathcal{L}_p\sbr{s\bar{u}^++t\bar{u}^-}=\mathcal{L}_p\sbr{s\bar{u}^+ }+\mathcal{L}_p\sbr{ t\bar{u}^-}.$$
		Similar to \eqref{s3}, we can easily check that $(1,1)$ is the unique maximum point of $f(s,t)$. If $(s,t)\neq (1,1)$, we have
		\begin{equation}
			\mathcal{L}_p\sbr{s\bar{u}^++t\bar{u}^-+\varepsilon\eta(s,t)\varphi}\leq  \mathcal{L}_p\sbr{s\bar{u}^++t\bar{u}^-} <f(1,1)=\mathcal{L}_p\sbr{ \bar{u}}=\mathcal{B}_p.
		\end{equation}
		If $(s,t)=(1,1)$, then $\eta(1,1)=1$ and
		\begin{equation}
			\mathcal{L}_p\sbr{s\bar{u}^++t\bar{u}^-+\varepsilon\eta(s,t)\varphi}\leq  f(1,1)- \frac{1}{2}\sigma\varepsilon\eta(1,1)<  \mathcal{B}_p.
		\end{equation}
		That is, for all $(s,t)\in \sbr{\R^+}^2$, we have
		\begin{equation}\label{s4}
			\mathcal{L}_p\sbr{s\bar{u}^++t\bar{u}^-+\varepsilon\eta(s,t)\varphi}<  \mathcal{B}_p.
		\end{equation}
		Let $g(s,t)=s\bar{u}^++t\bar{u}^- $,  $h(s,t)=s\bar{u}^++t\bar{u}^-+\varepsilon\eta(s,t)\varphi$, and
		\begin{equation}
			\vec{\Psi}_0(s,t)= \sbr{\Psi_0^+(s,t),\Psi_0^-(s,t)}:=\sbr{\mathcal{L}_p^\prime(g(s,t))g(s,t)^+,\mathcal{L}_p^\prime(g(s,t))g(s,t)^-}.
		\end{equation}
		\begin{equation}
			\vec{\Psi}_1(s,t)= \sbr{\Psi_1^+(s,t),\Psi_1^-(s,t)}:=\sbr{\mathcal{L}_p^\prime(h(s,t))h(s,t)^+,\mathcal{L}_p^\prime(h(s,t))h(s,t)^-}.
		\end{equation}
		Set  $D:=(\delta,2-\delta)\times (\delta,2-\delta)$ for $0<\delta<1-\varepsilon$. Then it follows from \eqref{f34} that $\vec{\Psi}_0=\vec{\Psi}_1$  on $\partial D$. Since  $(1,1)$ is the unique global maximum point of $f(s,t)$, we have  $\mathbf{0}=(0,0)\not \in\vec{\Psi}_0(\partial D) $ and $\vec{\Psi}_0(1,1)=\mathbf{0}$. By the Brouwer degree theory, it follows that
		\begin{equation}
			\deg (\vec{\Psi}_1,D, \mathbf{0}) =	\deg (\vec{\Psi}_0,D, \mathbf{0})=1.
		\end{equation}
		Therefore, there exists $ (s_0,t_0)\in D$ such that  $\vec{\Psi}_1(s_0,t_0)=0$. That is, $h(s_0,t_0) \in \mathcal{M}_p$. Combining this with \eqref{s4}, one gets
		$	\mathcal{B}_p\leq \mathcal{L}_p(h(s_0,t_0) )<\mathcal{B}_p$, which is a contradiction. From this, we can conclude that $\bar{u}$ is a critical point of $\mathcal{L}_p$, and hence is a weak solution of the problem \eqref{System2}. This completes the proof.
	\end{proof}
	\begin{proposition}\label{Lemma energy estimate2}
		Using the notations above,  we have
		\begin{equation}
			\limsup_{p\to 2^*} \mathcal{B}_p\leq \mathcal{B},
		\end{equation}
		where  $\mathcal{B} $, $\mathcal{B}_p$   are defined in \eqref{Defi of energy },   \eqref{Defi of energy-p} respectively.
	\end{proposition}
	
	\begin{proof}
		For any $\varepsilon>0$, we  can  find $u\in \mathcal{M}$ such that $\mathcal{L}(u)\leq \mathcal{B}+\varepsilon$ and $\mathcal{G}(u^{\pm})= 0$. Then  there exists $H>0$ such that $\mathcal{G}(Hu^{\pm})\leq -1$. Since $p \to 2^*$, there exists $\delta>0$ such that
		\begin{equation}
			|\mathcal{L}(hu^{\pm})-\mathcal{L}_p(hu^{\pm})|+|\mathcal{G}(hu^{\pm})-\mathcal{G}_p(hu^{\pm})|\leq \varepsilon, ~\text{ for any }~ 2^*-\delta<p<2^*~ \text{ and }~ 0<h\leq H.
		\end{equation}
		Then we have $\mathcal{G}_p(Hu^{\pm})\leq -\frac{1}{2}$ for any $2^*-\delta<p<2^*$, it is easy to see that there exist $s,t\in  \sbr{0,H} $ such that $\mathcal{G}_p(su^{+})=\mathcal{G}_p(tu^{-})=0$. That is,
		\begin{equation}
			\tilde{u}:=su^{+}+tu^{-}\in\mathcal{M}_p.
		\end{equation}
		Since $u\in \mathcal{M}$, one can easily check that $(1,1)$ is the unique maximum point of  $f(s,t)=\mathcal{L}( su^{+})+\mathcal{L}( tu^{-})$.
		Then for any $2^*-\delta<p<2^*$, we have
		\begin{equation}
			\begin{aligned}
				\mathcal{B}_p &\leq \mathcal{L}_p(\tilde{u}) =\mathcal{L}_p( su^{+}) +\mathcal{L}_p( tu^{-})  \leq \mathcal{L} ( su^{+}) +\mathcal{L} ( tu^{-})+2\varepsilon \leq \mathcal{L}(u)+2\varepsilon\leq \mathcal{B}+3\varepsilon.
			\end{aligned}
		\end{equation}
		This completes the proof since $\varepsilon$ is arbitrary.
	\end{proof}
	
	\begin{proof}[ \bf Proof of Theorem \ref{Theorem 1} ]
		By Proposition \ref{Lemma energy estimate1} and \ref{Lemma energy estimate2}, we have
		\begin{equation}\label{s5}
			\limsup_{p\to 2^*} \mathcal{B}_p\leq \mathcal{B}<\mathcal{C}+\frac{1}{N}\mathcal{S}^{\frac{N}{2}}.
		\end{equation}
		Taking $p_n\in\sbr{2,2^* }$,
		we  deduce from Proposition \ref{Lemma achieved p} that    problem \eqref{System2} has a sign-changing solution $ u_{n}\in \mathcal{M}_{p_n}$ with $\mathcal{L}_{p_n}(u_{n})=\mathcal{B}_{p_n}$ for each $n\in \mathbb{N}$,  and
		\begin{equation}\label{s6}
			\int_{\Omega}	 \nabla u_n\cdot \nabla\varphi =\lambda \int_{\Omega}  u_n\varphi+
			\int_{\Omega}  |u_n|^{p_n-2}u_n \varphi+ \theta \int_{\Omega} u_n\varphi \log u_n^2 ~\text{ for }~ \varphi\in C_0^{\infty}(\Omega).
		\end{equation}
		Letting $p_n\to 2^*$ as $n\to \infty$, by a similar proof to that in Proposition \ref{Lemma achieved p}, we can see from  $\eqref{s5}$  that $\lbr{u_n}$ is bounded in $H_0^1(\Omega)$. Hence, passing to subsequence, we may assume that there exists $u\in H_0^1(\Omega)$ such that
		\begin{equation}\label{s7}
			\begin{aligned}
				&	u_n \rightharpoonup u \text{ weakly in } H_0^1(\Omega),\\
				&u_n \to u \ \text{ strongly in } L^{p}(\Omega) \text{ for } 2\leq p<2^*,\\
				&u_n \to u \ \text{ almost everywhere in } \Omega.
			\end{aligned}
		\end{equation}
		Since $u_n$ satisfies \eqref{s6}, one can easily show that $\mathcal{L}^{\prime}(u)\varphi=0$ for $\varphi\in C_0^{\infty}(\Omega)$ by taking the limit $n\to \infty$.
		Set $w_n=u_n-u$.  Then by the Br\'{e}zis-Lieb Lemma  \cite{Brezis Lieb lemma}  and \cite[Lemma 2.3]{Deng-He-Pan=Arxiv=2022}, we have
		\begin{equation}\label{f9}
			\begin{aligned}
				&\int_{\Omega}|u_n^+|^2 =\int_{\Omega}|u^+|^2 +\int_{\Omega}|w_n^+| ^2+o_n(1) ~\text{ and } ~\int_{\Omega} u_n^2\log u_n^2=\int_{\Omega} u ^2\log u ^2 +o_n(1).
			\end{aligned}
		\end{equation}
		Then one can easily check the following
		\begin{equation}\label{s10}
			\begin{aligned}
				\mathcal{B}_{p_n}&=	\mathcal{L}_{p_n}(u_n)=\mathcal{L}_{p_n}(u_n)-\frac{1}{p_n}\mathcal{L}_{p_n}^{\prime}(u_n)u_n\\
				&= \sbr{\frac{1}{2}-\frac{1}{p_n}}\int_{\Omega}|\nabla u_n|^2-\sbr{\frac{1}{2}-\frac{1}{p_n}}\lambda\int_{\Omega}|  u_n|^2+\frac{\theta}{2}\int_{\Omega}|  u_n|^2-\sbr{\frac{1}{2}-\frac{1}{p_n}}\theta\int_{\Omega} u_n^2\log u_n^2\\
				&= \mathcal{L}(u)+\frac{1}{N} \int_{\Omega} |\nabla w_n|^2+o_n(1).
			\end{aligned}
		\end{equation}
		Now we consider the following cases.
		\vskip 0.1in
		{\bf Case 1:}   $u^+\equiv 0$ and  $u^-\equiv 0$.
		
		In this case,  we can see from \eqref{s7} and  $ u_{n}\in \mathcal{M}_{p_n}$   that
		\begin{equation}\label{s8}
			\begin{aligned}
				\int_{\Omega} |\nabla u_n^+|^2  &=\lambda \int_{\Omega} |u_n^+|^2+\int_{\Omega} |u_n^+|^{p_n}+\theta\int_{\Omega}(u_n^+)^2\log (u_n^+)^2\\
				&	\leq  \lambda \int_{\Omega} |u_n^+|^2 +\frac{p_n-q_0}{2^*-q_0} \int_{\Omega} |u_n|^{2^*}+\frac{2^*-p_n}{2^*-q_0} \int_{\Omega} |u_n^+|^{q_0}+\theta C_{q_0} \int_{\Omega}|u_n^+|^{q_0}\\
				&\leq \int_{\Omega} |u_n^+|^{2^*} +o_n(1),
			\end{aligned}
		\end{equation}
		where $0<q_0<2^*-\delta $ for some positive small constant $\delta$. Then we deduce  from the Sobolev embedding theorem
		that
		\begin{equation}
			\int_{\Omega} |u_n^+|^{2^*}\geq \mathcal{S}^{\frac{N}{2}}+o_n(1),
		\end{equation}
		and  therefore
		\begin{equation}\label{s9}
			\begin{aligned}
				\liminf_{n\to \infty} \mathcal{L}_{p_n}(u_{ n}^+)&=	\liminf_{n\to \infty}\mbr{\sbr{\frac{1}{2}-\frac{1}{p_n}}\int_{\Omega}|u_n^+ |^{p_n}+\frac{\theta}{2}\int_{\Omega}|u_n^+|^{2} }\\
				&\geq \liminf_{n\to \infty} \sbr{\frac{1}{2}-\frac{1}{2^*}}\int_{\Omega}|u_n^+ |^{2^*}\geq \frac{1}{N}\mathcal{S}^{\frac{N}{2}}.
			\end{aligned}
		\end{equation}
		Recalling that  $w_n^+:=u_n^+- u^+=u_n^+.$
		By using \eqref{s8} again, we have
		\begin{equation}
			\int_{\Omega} |\nabla w_n^+|^2  \leq \int_{\Omega} |w_n^+|^{2^*} +o_n(1)\leq \mathcal{S}^{-\frac{2^*}{2}}\sbr{\int_{\Omega} |\nabla w_n^+|^2}^{\frac{2^*}{2}} +o_n(1).
		\end{equation}
		From this we deduce that either $\int_{\Omega} |\nabla w_n^+|^2=o_n(1)$ or  $\liminf\limits_{n\to \infty}	\int_{\Omega} |\nabla w_n^+|^2  \geq \mathcal{S}^{\frac{N}{2}}$ holds.
		
		For the case $\int_{\Omega} |\nabla w_n^+|^2=o_n(1)$, we have $u_n^+\to  0$ strongly in $H_0^1(\Omega)$, then $\mathcal{L}_{p_n}(u_n^+)\to 0$ as $n\to \infty$, which contradicts to \eqref{s9}. Therefore, we have $\liminf\limits_{n\to \infty}	\int_{\Omega} |\nabla w_n^+|^2  \geq \mathcal{S}^{\frac{N}{2}}$. Similarly, we have $\liminf\limits_{n\to \infty}	\int_{\Omega} |\nabla w_n^-|^2  \geq \mathcal{S}^{\frac{N}{2}}$. Hence, \eqref{s10} implies that
		\begin{equation}
			\liminf\limits_{n\to \infty}	\mathcal{B}_{p_n} \geq  \frac{1}{N} \liminf\limits_{n\to \infty}\int_{\Omega} |\nabla w_n^+|^2+ \frac{1}{N} \liminf\limits_{n\to \infty}\int_{\Omega} |\nabla w_n^-|^2 \geq \frac{2}{N}\mathcal{S}^{\frac{N}{2}}>\mathcal{C} + \frac{1}{N}\mathcal{S}^{\frac{N}{2}},
		\end{equation}
		which contradicts to \eqref{s5}. Therefore, Case 1 is impossible.
		
		\vskip 0.1in
		{\bf Case 2:}   $u^+ \not \equiv 0$, $u^-\equiv 0$ or   $u^+  \equiv 0$, $u^-\not\equiv 0$.

		Without loss of generality, we assume that $u^+ \not \equiv 0$  , $u^-\equiv 0$. Since $\mathcal{L}^{\prime}(u)\varphi=0$ for $\varphi\in C_0^{\infty}(\Omega)$,   then $u^+$ is a  weak solution of \eqref{System1} and therefore $u^+\in\mathcal{N}$. Similar to Case 1, we can prove that $\liminf\limits_{n\to \infty}	\int_{\Omega} |\nabla w_n^-|^2  \geq \mathcal{S}^{\frac{N}{2}}$. Then by using  \eqref{s10}  again, we have
		\begin{equation}
			\liminf\limits_{n\to \infty}	\mathcal{B}_{p_n} \geq  \mathcal{L}(u^+)+ \frac{1}{N} \liminf\limits_{n\to \infty}\int_{\Omega} |\nabla w_n^-|^2  \geq \mathcal{C} + \frac{1}{N}\mathcal{S}^{\frac{N}{2}},
		\end{equation}
		which  also contradicts to \eqref{s5}. Therefore, Case 2 is impossible.
		
		Since Case 1  and  Case 2 are both impossible,  we know that $u^+\not \equiv0$ and  $u^-\not \equiv0$. By $\mathcal{L}^{\prime}(u) =0$, we have $\mathcal{L}^{\prime}(u)u^+=\mathcal{L}^{\prime}(u)u^-=0$, yielding that $u\in \mathcal{M}$. Then $ 	\mathcal{L}(u)\geq \mathcal{B}  $. On the other hand, by the weak lower semi-continuity of norm, we deduce from Proposition \ref{Lemma energy estimate2} that
		\begin{equation}
			\begin{aligned}
				\mathcal{L}(u)&=\mathcal{L}(u)-\frac{1}{2}\mathcal{L}^{\prime}(u)u= \sbr{\frac{1}{2}-\frac{1}{2^*}}\int_{\Omega}|u|^{2^*}+\frac{\theta}{2}\int_{\Omega}|u|^{2}\\
				&\leq \liminf_{n\to \infty}\mbr{\sbr{\frac{1}{2}-\frac{1}{p_n}}\int_{\Omega}|u_n |^{p_n}+\frac{\theta}{2}\int_{\Omega}|u_n|^{2} }=\liminf_{n\to \infty}\mathcal{B}_{p_n} \leq \mathcal{B}.
			\end{aligned}
		\end{equation}
		That is, $\mathcal{L}(u)=\mathcal{B}$. By using the same arguments as used in that of  Proposition \ref{Lemma achieved p}, we can   prove that $u$ is a  sign-changing solution of problem \eqref{System1}.
		
		Finally, we claim that $u$ has exactly two nodal domains by following the strategies in  Castro et al \cite{Castro=1997=Rocky Mountain J M}.  Assume that $u$ has three nodal domains, $D_1$, $D_2$ and $D_3$. Moreover, we assume that $D_1$,$D_2$ are positive nodal domains, $D_3$   is a negative nodal domain. Then $u|_{D_1\cup D_3}\in \mathcal{M} $ and $u|_{  D_2}\in \mathcal{N} $. Therefore, we have  $\mathcal{L}(u)\geq \mathcal{B}+\mathcal{C} $, which contradicts to the fact  $\mathcal{L}(u)=\mathcal{B}$. Hence, the claim is true.

	\end{proof}
	
	\section{ The proof of Theorem \ref{Theorem 2}}\label{Sect4}
	In this section, we construct the infinitely many radial nodal solutions for equation \eqref{System1} when $\Omega$ is a ball.
	From now on, we assume that $\Omega=B_R:=\lbr{ x\in \R^N:|x|\leq R}$. Let  $H_{rad}(\Omega)$ be the subspace of  $H_0^1(\Omega)$ consisting of radial functions.
	
	As in  \cite{Cerami-Solimini-Struwe 1986,Struwe=1976=Arch math}, for any fixed positive integer $k\in \N$, we define
	\begin{equation}
		\begin{aligned}
			\mathcal{M}_{k^\pm}:=  &\left\lbrace u\in H_{rad}(\Omega): \exists~ 0=:r_0<r_1<r_2<\ldots<r_k:=R, \text{ such that  for}~ 1\leq j\leq k, \right.\\
			&\left.  u(r_j)=0 ~, \pm (-1)^{j-1}u(x)\geq 0,  u\not \equiv 0~\text{ in } ~\Omega_j, \text{ and }~ \mathcal{L}^\prime(u^{(j)})u^{(j)}=0   \right\rbrace,
		\end{aligned}
	\end{equation}
	where $	\Omega_1:=\lbr{x\in B_R:|x|\leq r_1} $ is the first node of $u$, and $	\Omega_j:=\lbr{x\in B_R:r_{j-1}< |x|\leq r_j} $ is the j-th node of $u$ for $2\leq j\leq k$. Here, let $u^{(j)}:=u$ in $\Omega_j$ and $u^{(j)}:=0$ outside $\Omega_j$.

	Observe that the sets $\mathcal{M}_{k^\pm} $ are not empty for any $k\in \N$. Without loss of generality, we may focus on the case of $\mathcal{M}_{k^+}$ and redefine $\mathcal{M}_{k^+}$ as $\mathcal{M}_{k} $ for the sake of convenience. The case of $\mathcal{M}_{k^-}$ can be treat similarly with some slight modifications.
	
	Define
	\begin{equation}
		\mathcal{B}_k:= \inf_{u\in \mathcal{M}_k } \mathcal{L}(u),~~ \text{ for } k\in \N.
	\end{equation}
	Then we have the following
	\begin{lemma} \label{Lemma N nodes energy estimate}
		Under the hypotheses of Theorem \ref{Theorem 2}. Assume that there exists $v\in \mathcal{M}_k$ satisfying $\mathcal{L}(v)=\mathcal{B}_k$, then for $k\in \N$, we have
		\begin{equation}
			\mathcal{B}_{k+1}<\mathcal{B}_k+\frac{1}{N}\mathcal{S}^{\frac{N}{2}}.
		\end{equation}
	\end{lemma}
	\begin{proof}
		The proof is inspired by \cite{Cerami-Solimini-Struwe 1986}, we give the details here for the reader's convenience. Let $v\in \mathcal{M}_k$ satisfy  $\mathcal{L}(v)=\mathcal{B}_k$. Denote   the first node of $v$ by $B_{r_1}:=\lbr{x\in B_R:|x|\leq r_1} $ and define the energy level by
		\begin{equation}
			\begin{aligned}
				\mathcal{C}  ({r_1})=\inf &\biggl\{ \mathcal{L}(u): u\in H_{rad}(B_{r_1}), u\geq 0, u\not \equiv 0 \text{ in } B_{r_1}, \text{ and } \\
				&\int_{B_{r_1}} |\nabla u|^2=\lambda\int_{B_{r_1}}|u|^2+ \int_{B_{r_1}}|u|^{2^*}+\theta\int_{B_{r_1}}u^2\log u^2 \biggr\}.
			\end{aligned}
		\end{equation}
		Then  we have
		\begin{equation}\label{s11}
			\mathcal{C}  ({r_1}) \leq \mathcal{L}(v|_{B_{r_1}}).
		\end{equation}
		
		By Theorem \ref{Theorem 1}, there exists a sign-changing solution $w$ of equation \eqref{System1} in $H_{rad}(B
		_{r_1})$ and for some $0<r<r_1$, such that $w(x)>0$, if $0<|x|<r$,  $w(x)<0$, if $ |x|>r$, and $w(x)=0$, if $|x|=r$. Moreover, by Proposition \ref{Lemma energy estimate1}, we have
		\begin{equation}\label{s12}
			\mathcal{L}(w)<\mathcal{C}  ({r_1}) +\frac{1}{N} \mathcal{S}^{\frac{N}{2}}.
		\end{equation}
		Now we define
		\begin{equation}
			\bar{w}(x):=
			\begin{cases}
				w(x), &\text{ if } x \in B_{r_1},\\
				-v(x), &\text{ if }   x\in  B_R\setminus  B_{r_1}.
			\end{cases}
		\end{equation}
		Then we can easily see that $\bar{w} \in \mathcal{M}_{k+1}$, and by \eqref{s11}, \eqref{s12}, we have
		\begin{equation}
			\mathcal{B}_{k+1} \leq \mathcal{L}(\bar{w})=\mathcal{L}(w)+\mathcal{L}(v|_{B_R\setminus  B_{r_1}})< \mathcal{L}(v)+ \frac{1}{N} \mathcal{S}^{\frac{N}{2}}=	\mathcal{B}_{k }+ \frac{1}{N} \mathcal{S}^{\frac{N}{2}}.
		\end{equation}
		This completes the proof.
	\end{proof}
	Now we proceed to prove Theorem \ref{Theorem 2}.
	\begin{proof}[\bf Proof of Theorem \ref{Theorem 2}]
		The main ideas of this proof come from \cite{Cerami-Solimini-Struwe 1986,W.Shuai=Nonlineariyu=2019}, but  the presence of critical and logarithmic terms makes the  situations    different. We will prove by mathematical induction on the number of  the nodes  in  the equation \eqref{System1} that for each $k\in \N$, there exists $u\in \mathcal{M}_k$ such that
		\begin{equation}\label{s17}
			\mathcal{L}(u)=\mathcal{B}_k.
		\end{equation}
		By \cite{Deng-He-Pan=Arxiv=2022}, the conclusion is obviously true for $k=1$. Let $k>1$ and  suppose by induction  hypothesis that  Theorem \ref{Theorem 2} holds true for $k-1$. Let $\lbr{ u_n }\subset \mathcal{M}_{k}$ be a minimizing sequence such that
		\begin{equation}
			\lim_{n\to \infty} \mathcal{L}(u_n)=\mathcal{B}_k.
		\end{equation}
		Then for any $n\in \N$. there exists $k$ nodal domains of $u_n$, that is,
		\begin{equation}
			0=:r_0^n<r_1^n<r_2^n<\ldots<r_k^n:=R.
		\end{equation}
		Passing to a subsequence if necessary, we may assume that
		\begin{equation}
			r_j:=\lim_{n\to \infty} r_j^n, \quad \text{ for all }  ~   j=0,1,\ldots,k.
		\end{equation}
		Obviously, we have
		\begin{equation}
			0=r_0\leq r_1\leq r_2\leq \ldots\leq r_k=R.
		\end{equation}
		\medbreak
		{\bf Step 1.} We  prove that all the inequalities above are  in fact  strict.
		\medbreak
		Suppose by contradiction that there exists $i\in \lbr{1,2,\ldots,k}$ such that
		\begin{equation}
			r_{i-1}=r_i.
		\end{equation}
		For any $n\in \N$, we define  $u_n^{(j)}:=u_n$ in $\Omega_j^n$ and $u^{(j)}_n:=0$ outside $\Omega_j^n$, where $\Omega_j^n=\lbr{x\in \Omega: r_{j-1}^n<|x|\leq r_{j }^n }$. Then we have meas$(\Omega_i^n)\to 0$ as $n\to \infty$, and  from the H\"{o}lder's inequality  we know that for any $q\in \left[ 2,2^*\right) $,
		\begin{equation}\label{s13}
			|u_n^{(i)}|_q/|u_n^{(i)}|_{2^*}\to 0 , ~~\text{ as } ~ n \to \infty.
		\end{equation}
		However, since $u_n\in \mathcal{M}_k$,   we deduce from the Sobolev embedding theorem and the inequality $s^2 \log s^2 \leq (2^*-1)^{-1}e^{-1}s^{ 2^*}, s\in \sbr{0,+\infty} $ that
		\begin{equation}\label{s14}
			\begin{aligned}
				\mathcal{S}\sbr{\int_{\Omega_i^n}|u^{(i)}_n |^{2^*} }^{\frac{2}{2^*}}& \leq \int_{\Omega_i^n} |\nabla u^{(i)}_n|^2= \int_{\Omega_i^n}|u^{(i)}_n|^{2^*}+\theta\int_{\Omega_i^n}(u^{(i)}_n)^2\log \sbr{e^{\frac{\lambda }{\theta }}(u^{(i)}_n)^2 }\\& \leq C \int_{\Omega_i^n}|u^{(i)}_n |^{2^*} ,
			\end{aligned}
		\end{equation}
		where $C=C(\lambda,\theta,N)>0$ is  a constant independent of $n$. By \eqref{s13} we know that  $|u_n^{(i)}|_q=o_n(1)$ for any $q\in \left[ 2,2^*\right) $.  Using the inequality $s^2\log s^2\leq C_q s^{q}$ for $s>0$ and  $q\in \left(  2,2^*\right) $, we also know that  $$\int_{\Omega_i^n} (u^{(i)}_n)^2\log   (u^{(i)}_n)^2  =o_n(1). $$
		Combining these with \eqref{s14}, we can easily see that
		\begin{equation}
			\liminf_{n\to \infty}\int_{\Omega_i^n}|u^{(i)}_n |^{2^*}  \geq \mathcal{S}^\frac{N}{2}.
		\end{equation}
		Then  by  Lemma \ref{Lemma N nodes energy estimate},
		\begin{equation}
			\begin{aligned}
				K:=	\liminf_{n\to \infty}\mathcal{L}(u^{(i)}_n)&=\liminf_{n\to \infty}\mbr{\mathcal{L}(u^{(i)}_n)-\frac{1}{2}\mathcal{L}^\prime(u^{(i)}_n)u^{(i)}_n}\\&=\liminf_{n\to \infty} \mbr{\frac{ 1}{N}\int_{\Omega_i^n}|u^{(i)}_n|^{2^*}+\frac{\theta}{2}\int_{\Omega_i^n}|u^{(i)}_n|^{2}}
				\\& \geq \frac{ 1}{N}\mathcal{S}^\frac{N}{2}>\mathcal{B}_{k}-\mathcal{B}_{k-1}.
			\end{aligned}
		\end{equation}
		Take $\varepsilon=\sbr{K-\mathcal{B}_{k}+\mathcal{B}_{k-1}}/2>0$, and choose $n_0\in \N$ such that
		\begin{equation}
			|\mathcal{L}(u^{(i)}_{n_0})-K|<\varepsilon, ~\text{ and }~	|\mathcal{L}(u_{n_0})-\mathcal{B}_k|<\varepsilon.
		\end{equation}
		Let
		\begin{equation}
			\hat{u}(x)=\begin{cases}
				u^{(j)}_{n_0}, &\quad  \forall x\in \Omega_j^{n_0}, ~\text{ if } j<i,\\
				0, &\quad \forall x\in \Omega_i^{n_0},\\
				-	u^{(j)}_{n_0}, &\quad  \forall x\in \Omega_j^{n_0}, ~\text{ if } j>i.
			\end{cases}
		\end{equation}
		Then  $	\hat{u}\in \mathcal{M}_{k-1}$ and therefore
		\begin{equation}
			\begin{aligned}
				\mathcal{B}_{k-1}\leq \mathcal{L}(	\hat{u})&=\mathcal{L}(u_{n_0})-\mathcal{L}(u^{(i)}_{n_0})
				<\mathcal{B}_{k }-K+2\varepsilon =\mathcal{B}_{k-1},
			\end{aligned}
		\end{equation}
		which is a contradiction. Therefore, we have
		\begin{equation}
			0=r_0< r_1<r_2< \ldots<r_k=R.
		\end{equation}	
		Define
		\begin{equation}
			\begin{aligned}
				&\bar{\Omega}_1:=\lbr{x\in B_R:|x|\leq r_1}\\
				&	\bar{\Omega}_j:=\lbr{x\in B_R:r_{j-1}< |x|\leq  r_j}, \quad j=2,\ldots,k,
			\end{aligned}
		\end{equation}
		and  by $\bar{u}_j$ the positive function of \eqref{System1} on $\bar{\Omega}_j$ such that
		\begin{equation}
			\begin{aligned}
				\mathcal{L}(\bar{u}_j)= \mathcal{C}  (\bar{\Omega}_j )=\inf_{u\in \mathcal{N}(\bar{\Omega}_j)}   \mathcal{L}(u),
			\end{aligned}
		\end{equation}
		where
		\begin{equation}
			\mathcal{N}(\bar{\Omega}_j):=\biggl\{: u\in H_{rad}(\bar{\Omega}_j ):  u\geq 0, u\not \equiv 0 \text{ in } \bar{\Omega}_j ,
			\int_{\bar{\Omega}_j } |\nabla u|^2=\lambda\int_{\bar{\Omega}_j }|u|^2+ \int_{\bar{\Omega}_j }|u|^{2^*}+\theta\int_{\bar{\Omega}_j }u^2\log u^2 \biggr\}.
		\end{equation}
		Now we construct $u\in  H_{rad}( B_R)$ as follows
		\begin{equation}\label{s15}
			u(x):=(-1)^{j-1}\bar{u}_j(x), \quad \forall x\in \bar{\Omega}_j, ~~ j=1,2,\ldots,k.
		\end{equation}
		Then one can easily see that $u\in \mathcal{M}_k$, and $\mathcal{L}(u)\geq \mathcal{B}_k$.
		\medbreak
		{\bf Step2.} We claim that $\mathcal{L}(u)\leq \mathcal{B}_k$.
		\medbreak
		Recall  that $\lbr{ u_n }\subset \mathcal{M}_{k}$ is  the minimizing sequence satisfying
		\begin{equation}
			\lim_{n\to \infty} \mathcal{L}(u_n)=\mathcal{B}_k.
		\end{equation}
		For every $j\in \lbr{1,\ldots,k}$ fixed and $n\in\N$, we define
		\begin{equation}
			v_n^{(j)}(x):=u_n^{(j)}(\frac{r_j^n}{r_j}x), \quad \forall x\in \bar{\Omega}_j.
		\end{equation}
		Then there exists $s_n>0$ such that $w_n^{(j)}(x):=s_nv_n^{(j)}(x)\in \mathcal{N}(\bar{\Omega}_j)$, which implies that $\mathcal{L}(w_n^{(j)})\geq \mathcal{C}  (\bar{\Omega}_j )$ and
		\begin{equation}
			\begin{aligned}
				0&=	\int_{\bar{\Omega}_j } \sbr{|\nabla w_n^{(j)}|^2-\lambda|w_n^{(j)}|^2- |w_n^{(j)}|^{2^*}-\theta (w_n^{(j)})^2\log (w_n^{(j)})^2	} d x\\
				&= \int_{\Omega_j^n} \mbr{\sbr{\frac{s_nr_j^n}{r_j}}^2|\nabla u_n^{(j)}|^2-s_n^2\lambda |u_n^{(j)} |^2-s_n^{2^*}|u_n^{(j)}|^{2^*}-\theta (s_nu_n^{(j)})^2\log (s_nu_n^{(j)})^2}\sbr{\frac{ r_j}{r_j^n} }^N dx.
			\end{aligned}
		\end{equation}
		It follows from $u_n^{(j)} \in \mathcal{N}(\Omega_j^n)$ that
		\begin{equation}
			0=\int_{\Omega_j^n} \mbr{|\nabla u_n^{(j)}|^2-\lambda |u_n^{(j)} |^2-|u_n^{(j)}|^{2^*}-\theta (u_n^{(j)})^2\log (u_n^{(j)})^2}dx.
		\end{equation}
		Combining this with the fact that   $r_j=r_j^n+o_n(1)$,  by a standard argument we have  $s_n=1+o_n(1)$, which implies that
		\begin{equation}
			\liminf_{n\to \infty}\mathcal{L}(u_n^{(j)})= \liminf_{n\to \infty}\mathcal{L}(w_n^{(j)})\geq \mathcal{C}  (\bar{\Omega}_j )=\mathcal{L}(\bar{u}_j), \quad \text{ for every } j=1,\ldots, k.
		\end{equation}
		Then by the construction of $u$ (see \eqref{s15}), we deduce that
		\begin{equation}
			\mathcal{B}_k=\lim_{n\to \infty} \mathcal{L}(u_n)=\lim_{n\to \infty}\sum_{j=1}^{k}\mathcal{L}(u_n^{(j)})\geq \sum_{j=1}^{k}\mathcal{L}(\bar{u}_j) =\mathcal{L}(u).
		\end{equation}
		\medbreak
		{\bf Step3.} We prove that the function  $u$  is a solution of equation \eqref{System1} satisfying  $\mathcal{L}(u)=\mathcal{B}_k$.
		\medbreak
		
		It follows from the the construction of $u$ and the conclusion of Step2  that
		
		\begin{equation}
			\mathcal{L}(u)=\mathcal{B}_k.
		\end{equation}
		Obviously, $u$ is radial. Then we set $|x|=r$  and rewrite $u(|x|)$ by $u(r)$. Then  $u(r)$ satisfies
		\begin{equation}\label{s16}
			-(r^{N-1}u')'=r^{N-1}(\lambda u+ |u|^{2^*-2}u+u\log u^2)
		\end{equation}
		on the set $\mathcal{U}:=\lbr{r\in \sbr{0,R}: r\neq r_j, ~j=1,\ldots,k-1}$,  where $'$ denotes $\frac{d}{dr}$. Moreover, by  \cite{Deng-He-Pan=Arxiv=2022}, we know that $u$ is of class $C^2$ on $\mathcal{U}$.
		
		It remains to show that $u(r)$ satisfies \eqref{s16} for all $r\in \sbr{0,R}$. For that purpose, we will
		show
		\begin{equation} \label{s18}
			\lim_{r\to r_j^+} u'(r)=\lim_{r\to r_j^-} u'(r), \text{ for all  }j\in\lbr{1,\ldots,k-1}.
		\end{equation}
		Suppose by contradiction that there exists some $i\in \lbr{ 1,\ldots,k-1}$ such that   $u'_+\neq u'_-$, where
		\begin{equation} \label{s20}
			u'_{\pm}: =	\lim_{r\to r_i^{\pm}} u'(r).
		\end{equation}
		We may assume that $u(r)\geq 0$ for $ r\in \mbr{r_{i-1},r_i}$, and $u(r)\leq 0$ for $ r\in \mbr{r_{i},r_{i+1}}$. Take $\delta>0$ small and define the function   $v:\mbr{r_{j-1},r_{j+1}} \to \R$ as following
		\begin{equation}
			v(r)=\begin{cases}
				u(r), &\text{ if }  |r-r_j|\geq \delta,\\
				u(r_j-\delta)+\cfrac{u(r_j+\delta)+u(r_j-\delta)}{2\delta}(r-r_j+\delta), &\text{ if }  |r-r_j|\leq \delta.
			\end{cases}
		\end{equation}
		It is easy to see that $v\in C[r_{j-1},r_{j+1}]$ and there exists $c_0=c_0(\delta)\in \sbr{r_{j}-\delta,r_j+\delta}$ such that $v(c_0)=0$ for each  $\delta>0$ small.  Then, there exist  $s=s(\delta)>0$ and $t=t(\delta)>0$ such that
		\begin{equation}
			sv^+\in \mathcal{N}(\mbr{r_{j-1},c_0}), \quad  tv^-\in \mathcal{N}(\mbr{ c_0,r_{j+1}}).
		\end{equation}
		Moreover, one can easily see that
		\begin{equation}\label{s21}
			\lim_{\delta\to0^+}s(\delta)=\lim_{\delta\to0^+} t(\delta)=1.
		\end{equation}
		Define
		\begin{equation}\label{s19}
			z(r)=\begin{cases}
				sv^+(r), &\text{ if  }~ r\in \mbr{r_{j-1},c_0},\\
				tv^-(r), &\text{ if } ~ r \in \mbr{ c_0,r_{j+1}}.
			\end{cases}
		\end{equation}
		Consider the following functional
		\begin{equation}
			\mathcal{L}_j(z):=\int_{r_{j-1}}^{r_{j+1}} \mbr{\frac{1}{2}( z')^2-\frac{\lambda}{2} |z|^2-\frac{1}{2^*}|z|^{2^* }-\frac{\theta}{2}z^2(\log z^2-1)}r^{N-1} ~dr.
		\end{equation}
		Since $s^2\log s^2-s^2\geq t^2\log t^2- t^2+(s^2-t^2)\log t^2$ for $st>0$, then
		\begin{equation}
			\begin{aligned}
				&\sbr{\int_{r_{j-1}}^{r_j-\delta}+\int_{r_j+\delta}^{r_{j+1}}}\mbr{\frac{1}{2}( z')^2-\frac{\lambda}{2} |z|^2-\frac{1}{2^*}|z|^{2^* }-\frac{\theta}{2}z^2(\log z^2-1)}r^{N-1} ~dr\\
				&\leq \sbr{\int_{r_{j-1}}^{r_j-\delta}+\int_{r_j+\delta}^{r_{j+1}}}\mbr{\frac{1}{2}( z')^2-\frac{\lambda}{2}  |z|^2-\frac{1}{2^*}|z|^{2^* }-\frac{\theta}{2}u^2(\log u^2-1)+\frac{\theta\sbr{u^2-z^2}}{2}\log u^2}r^{N-1} ~dr.
			\end{aligned}
		\end{equation}
		Since $ u$ satisfies equation \eqref{s16}, we have
		\begin{equation}
			\begin{aligned}
				\mathcal{L}_j(z)
				&\leq \mathcal{L}_j(u)+\sbr{\int_{r_{j-1}}^{r_j-\delta}+\int_{r_j+\delta}^{r_{j+1}}} \mbr{\frac{1}{2}( z')^2-\frac{\lambda}{2}  |z|^2-\frac{1}{2^*}|z|^{2^* }-\frac{\theta}{2}z^2 \log u^2-\frac{1}{N}|u|^{2^*}}r^{N-1} ~dr\\
				&+\int_{r_j-\delta}^{r_j+\delta} \mbr{\frac{1}{2}( z')^2-\frac{\lambda}{2} |z|^2-\frac{1}{2^*}|z|^{2^* }-\frac{\theta}{2}z^2(\log z^2-1)-\frac{1}{N}|u|^{2^*}-\frac{\theta}{2}u^2}r^{N-1} ~dr\\
				&=:\mathcal{L}_j(u)+I_1(\delta)+I_2(\delta).
			\end{aligned}
		\end{equation}
		Then we deduce from \eqref{s16} and \eqref{s19} that
		\begin{equation}\label{s24}
			\begin{aligned}
				&\int_{r_{j-1}}^{r_j-\delta}\mbr{\frac{1}{2}( z')^2-\frac{\lambda}{2}  |z|^2-\frac{1}{2^*}|z|^{2^* }-\frac{\theta}{2}z^2 \log u^2-\frac{1}{N}|u|^{2^*}}r^{N-1} ~dr\\
				&= \int_{r_{j-1}}^{r_j-\delta} \mbr{\frac{1}{2}s^2( u')^2-\frac{\lambda }{2}s^2|u|^2-\frac{s^{2^*}}{2^*}|u|^{2^* }-\frac{\theta}{2}s^2 u^2 \log u^2-\frac{1}{N}|u|^{2^*}}r^{N-1} ~dr\\
				&=\int_{r_{j-1}}^{r_j-\delta} \frac{s^2}{2}\mbr{ ( u')^2r^{N-1}+ (r^{N-1}u')u}  +\mbr{ \sbr{\frac{s^2}{2}-\frac{s^{2^*}}{2^*}} |u|^{2^*}-\frac{1}{N}|u|^{2^*}}r^{N-1}  ~dr\\
				&\leq \int_{r_{j-1}}^{r_j-\delta} \frac{s^2}{2}\mbr{ ( u')^2r^{N-1}+ (r^{N-1}u')u} ~dr
				=\frac{s^2}{2} (r_j-\delta)^{N-1}u'(r_j-\delta)u(r_j-\delta).
			\end{aligned}
		\end{equation}
		Since $u(r_j)=0$ and $(r^{N-1}u')'(r_j)=0$,  we have
		\begin{equation}
			u(r_j-\delta)=-\delta u_-'+o(\delta),  \quad (r^{N-1}u') (r_j-\delta)=r_j^{N-1}u_-'+o(\delta),
		\end{equation}
		where $u_-'$  is given by \eqref{s20}.  Then we can deduce from  \eqref{s21}, \eqref{s24} that
		\begin{equation}\label{s22}
			\begin{aligned}
				&	\int_{r_{j-1}}^{r_j-\delta}\mbr{\frac{1}{2}( z')^2-\frac{\lambda}{2}  |z|^2-\frac{1}{2^*}|z|^{2^* }-\frac{\theta}{2}z^2 \log u^2-\frac{1}{N}|u|^{2^*}}r^{N-1} ~dr \leq -\frac{ r_j^{N-1}\delta}{2} (u_-')^2+o(\delta).
			\end{aligned}
		\end{equation}
		Similarly, we can prove
		\begin{equation}\label{s23}
			\begin{aligned}
				&\int_{r_j+\delta}^{r_{j+1}} \mbr{\frac{1}{2}( z')^2-\frac{\lambda}{2}  |z|^2-\frac{1}{2^*}|z|^{2^* }-\frac{\theta}{2}z^2 \log u^2-\frac{1}{N}|u|^{2^*}}r^{N-1} ~dr \leq
				-  \frac{ r_j^{N-1}\delta}{2} (u_+')^2+o(\delta).
			\end{aligned}
		\end{equation}
		With these facts in mind, and observing that
		\begin{equation}
			\begin{aligned}
				\int_{r_j-\delta}^{r_j+\delta}( z')^2r^{N-1} ~dr&=\frac{\mbr{u(r_j+\delta)-u(r_j-\delta)}^2}{4\delta^2}\sbr{ \frac{(r_j+\delta)^N}{N}-\frac{(r_j-\delta)^N}{N}}\\&=\frac{r_j^{N-1}\delta}{2}\sbr{u_+'+ u_-'}^2+o(\delta),
			\end{aligned}
		\end{equation}
		and
		
		\begin{equation}
			\int_{r_j-\delta}^{r_j+\delta} \mbr{\frac{\lambda}{2} |z|^2+\frac{1}{2^*}|z|^{2^* }+\frac{\theta}{2}z^2(\log z^2-1)+\frac{1}{N}|u|^{2^*}+\frac{\theta}{2}u^2}r^{N-1} ~dr=o(\delta),
		\end{equation}
		we have
		\begin{equation}
			\mathcal{L}_j(z)\leq \mathcal{L}_j(u)-\frac{r_j^{N-1}\delta}{4}\sbr{u_+'- u_-'}^2+o(\delta)<\mathcal{L}_j(u),
		\end{equation}
		for $\delta >0$ small. However, by the construction of $u$ and $z$, we can see that $\mathcal{L}_j(u)\leq \mathcal{L}_j(z)$, which is a contradiction. Hence, we have \eqref{s18} holds, and the function  $u$  is a solution of equation \eqref{System1} satisfying  $\mathcal{L}(u)=\mathcal{B}_k$. We complete the proof.
	\end{proof}

\end{document}